\documentclass[18pt,reqno]{amsart}
\usepackage{graphicx} 
\usepackage{amsmath,amssymb,amsthm,hyperref}
\usepackage[labelfont=bf,labelsep=colon]{caption} 
\usepackage{stmaryrd}
\usepackage{enumerate}
\usepackage{tikz}
\usepackage{pstricks}
\usepackage{pstricks-add}
\usepackage{pst-plot}

\newtheorem{Theorem}{Theorem} 
 \newtheorem{theorem}{Theorem}[section]

 \newtheorem{lemma}[theorem]{Lemma}
 \newtheorem{proposition}[theorem]{Proposition}

 \theoremstyle{definition}
 \newtheorem{definition}[theorem]{Definition}

 \newcommand{\TT}{\mathcal{T}}
 \newcommand{\N}{\mathbb{N}}
 \renewcommand{\H}{\operatorname{H}}
 \newcommand{\HH}{\mathcal{H}}
 \renewcommand{\L}{\operatorname{L}}
 \newcommand{\LL}{\mathcal{L}}
 \newcommand{\dsng}{\operatorname{DNSG}}
 
 \newcommand{\RR}{\operatorname{R}}
 \newcommand{\balpha}{\boldsymbol{\alpha}}
  \newcommand{\bbeta}{\boldsymbol{\beta}}
 
\title{Non-separable graphs meet Ledoux's polynomials}
\author{Paul Mansanarez}
\address{
Paul Mansanarez, Nantes Université/Universit\'e libre de Bruxelles, France/Belgium. E-mail:
paul.mansanarez@ulb.be}

\begin{document}

\begin{abstract}
In the path-breaking article \cite{LED16}, an integral representation of the derivatives of the entropy along the heat flow of a probability measure was established under suitable moment conditions.
 
These integral representations have found significant applications in diverse domains — notably in information theory (e.g., entropy power inequalities, monotonicity of Fisher information) and in estimation theory (through the link between entropy derivatives and the minimum mean square error, MMSE, in Gaussian channels). 
 
The representations involve multivariate polynomials $(\RR_n)_n$, arising from a Lie algebra framework on multilinear operators. Despite their central role, the combinatorial structure of these polynomials remains only partially understood.
 
In this note, we prove that the number of monomials in $\RR_n$ coincides with the number of degree sequences with degree sum $2n$ having a non-separable graph realization, thereby resolving a conjecture from \cite{MPS24}, and drawing an interesting link between these two domains. 
\end{abstract}
\maketitle
\section{Introduction}
In this paper, we investigate the combinatorial structure of a specific sequence of multivariate polynomials
\[
\widetilde{\RR}_n = X_n^2 + \RR_n,
\]
where $\RR_n \in \mathbb{Z}[X_2, \ldots, X_{n-1}]$ is defined recursively by $\RR_2 = 0$ and the relation
\begin{align}
\RR_{n+1} = A_n + \L(\RR_n) + \H(\RR_n),
\label{f:ind_formulaR}
\end{align}
where
\[
A_n = -\sum_{k=1}^{n-1} \binom{n}{k} X_{1+k}X_{1+n-k}X_n,
\]
and where $\L$ and $\H$ are linear operators on multivariate polynomials, defined by
\begin{align}\label{def:L}
\L(X_{\alpha_1} \cdots X_{\alpha_r}) &= \sum_{1 \leq i < j \leq r} X_{\alpha_1} \cdots X_{\alpha_i + 1} \cdots X_{\alpha_j + 1} \cdots X_{\alpha_r},
\end{align}
and
\begin{align}\label{def:H}
\H(X_{\alpha_1} \cdots X_{\alpha_r}) &= -\frac{1}{2} \sum_{k=1}^r \sum_{l=1}^{\alpha_k - 1} \binom{\alpha_k}{l} X_{1 + l} X_{1 + \alpha_k - l} \prod_{\substack{i=1 \\ i \neq k}}^{r} X_{\alpha_i},
\end{align}
for all positive integers $\alpha_1,\ldots, \alpha_r$.

As examples, the first few polynomials are:
\begin{align*}
\RR_2 &= 0, \\
\RR_3 &= -2X_2^3,\\ 
\RR_4 &= -12X_2X_3^2 + 6X_2^4,\\
\RR_5 &= -20X_2X_4^2 - 30X_3^2X_4 + 120X_2^2X_3^2 - 24X_2^5, \\
\RR_6 &= -30X_2X_5^2 - 120X_3X_4X_5 + 900X_2X_3^2X_4 + 300X_2^2X_4^2 - 30X_4^3 \\
&\quad - 1200X_2^3X_3^2 + 210X_3^4 + 120X_2^6.
\end{align*}
More explicit expressions can be computed using the algorithm described in~\cite{AlgoR}.

These polynomials arise in the algebraic framework of $\Gamma$-calculus and iterated gradients (see~\cite{BAK13, LED95}). In particular, in~\cite[Theorem 2]{LED16}, using the formalism of~\cite{LED95}, the author express the $n$-th time derivative of the {\it entropy along the heat flow} as
\begin{align}
\partial_t^n H(X + \sqrt{2t}\,N) = (-2)^{n-1} \int_{\mathbb{R}} \widetilde{\RR}_n\big(u_t^{(2)}, \ldots, u_t^{(n)}\big)(x)\,\mathrm{d}x, \quad t > 0,
\end{align}
where $H$ is the Shannon entropy (see~\cite{COV}), $X$ and $N$ are independent random variables with $N \sim \mathcal{N}(0,1)$, and $u_t^{(i)}$ is the $i$-th derivative of the function $u_t$, which depends on $X$ and $t$. The relationship between $X$ and $X_t := X + \sqrt{2t}\,N$, the {\it heat flow} starting from $X$, is a fundamental problem in information and signal theory (see~\cite[Chapter 9]{COV} or~\cite{VerGuo} for an overview). Moreover, the entropy function $t \mapsto H(X + \sqrt{2t}\,N)$ has been the subject of several conjectures (see~\cite{LED20Rev}).

In a recent contribution~\cite{MPS24}, the authors extended the work of~\cite{LED16} by exploring further the combinatorial structure underlying the polynomials $(\widetilde{\RR}_n)_n$, and obtained new results related to the MMSE conjecture (see~\cite{LED20Rev} for an introduction). This conjecture states that the entropy function $t \mapsto H(X + \sqrt{2t}\,N)$ characterizes certain classes of distributions. However, the specific functional structure of these polynomials is not the focus of the present work. Indeed, as mentioned above, we are interested in the combinatorial properties of the polynomials $(\RR_n)_n$.

While performing a comprehensive study of the polynomials $(\RR_n)_n$, the authors of~\cite{MPS24} conjectured that the number of monomials in $\RR_n$ is equal to $d_{ns}(n) - 1$, where $d_{ns}(n)$ is the number of degree sequences of sum $2n$ that admit a non-separable graph realization (see~\cite[Appendix A]{MPS24} and subsection~\ref{sb:graphs} below for a background on graph theory).

In this paper, we prove their conjecture, namely:
\begin{proposition}\label{main}
Let $n$ be an integer greater than $2$. Then the number of terms in $\RR_n$ is equal to $d_{ns}(n) - 1$.
\end{proposition}
This result establishes a connection between the structure of these polynomials and the theory of non-separable graphs.

The remainder of the paper is organized as follows. In Section~\ref{s:prolegomena}, we recall some key results on non-separable graphs and introduce notations related to $\RR_n$ necessary for proving Proposition~\ref{main}. The proof of the proposition is then presented in the final section.

\section{Notations and preliminaries}\label{s:prolegomena}
In this section, we give some preliminaries and definitions on graphs, as well as notations for the study of the combinatorics of the coefficients of $R_n$. For more details, we refer the reader to \cite{RST,GYZ2003,Whit1931} for materials on non-separable graphs, and to \cite{MPS24} for materials on the polynomials $(\RR_n)_n$. 

In the rest of the article, for a finite set $E$, we will denote by $|E|$ its cardinal, that is, the number of elements of $E$.
\subsection{Degree sequences of non-separable graphs}\label{sb:graphs}

In this article, we will consider the same kind of  graphs as in \cite{RST}, that is, graphs that are finite, undirected, and do not contain loops (i.e. edges which start and end at the same vertex and are not adjacent to any other vertex). For sake of clarity, we provide a definition below.
\begin{definition}
    A \emph{graph} $G$ is a triple $(V,E,r)$, where
    \begin{itemize}
        \item $V$ is a finite set, called \emph{set of vertices} ;
        \item $E$ is a finite set, called \emph{set of edges} ;
        \item $r$ is an application $E \longrightarrow \bigl \{ \{x,y\}\; ; \; (x,y) \in V^2 \bigr \}$ that assigns to any edge $e \in E$, two endpoints in $V$.
    \end{itemize}
   Furthermore, we assume that $G$ has no loops, that is $\forall e \in E, \; |r(e)| = 2$.
\end{definition}

With that said, the graphs we consider may contain multiple edges; as such, they are sometimes referred to as {\it multigraphs}.

\begin{definition}
    Let $G = (V,E,r)$ be a graph. The \emph{degree} of a vertex $v \in V$ is the integer $d := |\{e \in E \; ; \; v \in r(e) \} |$.
    
    \noindent The \emph{degree sequence} of $G$ is the (finite) sequence formed by arranging the degrees of the vertices of $G$ in a non-decreasing order.
\end{definition}
Let us note that a degree sequence $(d_1, \ldots, d_n)$ can be associated to several different graphs.

In order to define what a non-separable graph is, we need the classical notions of subgraphs, in particular the subgraph $G-v$ of $G$ obtained by removing a vertex and all edges that has it as an endpoint, and the notion of component and connected graph. To keep the exposition short, we refer to \cite{GYZ2003} for these precise definitions, and we define non-separable graphs as follows:

\begin{definition}
Let $G = (V,E,r)$ be a graph. A vertex $v \in V$ is a \emph{cut-vertex} of $G$ if $|E| \geq 2$ and $G-v$ has more components than $G$.

\noindent The graph $G$ is said to be \emph{non-separable} if it is connected and has no cut-vertices.
\end{definition}  

In 1962, Hakimi \cite{HAK} characterized those degree sequences for which there exists a non-separable graph realization.  

\begin{Theorem}
\label{hakimi_char_nonsep}
Let $d_1 \geq d_2 \geq \dots \geq d_n \geq 2$ be integers with $n\geq 2.$  Then there exists a non--separable graph with degree sequence $(d_1, d_2, \dots, d_n)$ if and only if 
\begin{itemize}
\item{} $d_1+d_2+\dots+d_n $ is even and 
\item{} $d_1 \leq d_2+d_3+\dots +d_n -2n+4.$  
\end{itemize}
\end{Theorem}


In \cite{RST}, R\o dseth, Tverberg, and Sellers considered the function $d_{ns}(n)$ which counts the number of degree sequences of degree sum $2n$ with a non--separable graph realization.  Using both generating functions as well as bijective arguments, the authors of \cite{RST} proved the following:    

\begin{Theorem}
\label{new_nonsep}
For all $m\geq 2,$ 
\begin{align}
d_{ns}(2m) = p(2m) - p(2m-1) - \sum_{j=0}^{m-2} p(j),
\end{align}
where $p(k)$ is the number of unrestricted integer partitions of $k$.  
\end{Theorem}

\noindent 
So, for example, the number of degree sequences of sum 6 with non--separable graph realizations is 
\[d_{ns}(6) = p(6) - p(5) - p(0) - p(1) = 11 - 7 -1-1 = 2.\] 
The two partitions in question, along with corresponding non-separable graph realizations, are shown below in Figure \ref{part3}.
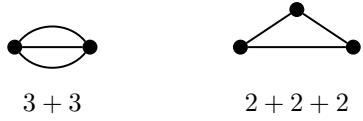
\begin{figure}[!htbp]
\centering
\begin{tikzpicture}[
  vertex/.style={circle,fill=black,inner sep=2pt},
  edge/.style={-Latex}
]
	
	\node[vertex] (a) at (0,0)  {};
    \node[vertex] (b) at (1,0)  {};

    \draw[line width=0.75pt, line cap=round] (a)--(b) ;

    \draw[line width=0.75pt, line cap=round] (a) to[bend left=50] (b);
    \draw[line width=0.75pt, line cap=round] (a) to[bend right=50] (b);
	\node at (0.5,-0.75) {$3+3$};
	
	\node[vertex] (a) at (3,0)  {};
    \node[vertex] (b) at (3.75,0.5)  {};
    \node[vertex] (c) at (4.5,0)  {};
    \draw[line width=0.75pt, line cap=round] (a)--(b) (b)--(c) (c)--(a);

	\node at (3.75,-0.75) {$2+2+2$};
	
\end{tikzpicture}
\caption{Non-separable graphs realizations of degree sequences of degree sum $6$}
\label{part3}
\end{figure}

One last example: the number of degree sequences of sum $8$ with non-separable graph realizations is
\[d_{ns}(8) = p(8) - p(7) - p(0) - p(1) - p(2) = 22 - 15 -1-1-2 = 3,\]
and the corresponding non-separable graph realizations are shown below in Figure \ref{part4}.
\vspace{2mm}
\begin{figure}[!htbp]
\centering
\begin{tikzpicture}[
  vertex/.style={circle,fill=black,inner sep=2pt},
  edge/.style={-Latex}
]
	
	\node[vertex] (a) at (0,0)  {};
    \node[vertex] (b) at (1,0)  {};


    \draw[line width=0.75pt, line cap=round] (a) to[bend left=50] (b);
    \draw[line width=0.75pt, line cap=round] (a) to[bend right=50] (b);
    \draw[line width=0.75pt, line cap=round] (a) to[bend left=20] (b);
    \draw[line width=0.75pt, line cap=round] (a) to[bend right=20] (b);
	\node at (0.5,-0.75) {$4+4$};
	
	\node[vertex] (a) at (3,0)  {};
    \node[vertex] (b) at (3.5,0.75)  {};
    \node[vertex] (c) at (4,0)  {};
    \draw[line width=0.75pt, line cap=round] (a)--(b) (b)--(c);
    \draw[line width=0.75pt, line cap=round] (a) to[bend left=20] (c);
    \draw[line width=0.75pt, line cap=round] (a) to[bend right=20] (c);

	\node at (3.5,-0.75) {$3+3+2$};

	\node[vertex] (a) at (6,0)  {};
    \node[vertex] (b) at (6,1)  {};
    \node[vertex] (c) at (7,0)  {};
    \node[vertex] (d) at (7,1)  {};
    \draw[line width=0.75pt, line cap=round] (a)--(b) (b)--(d) (d)--(c) (a)--(c);

	\node at (6.5,-0.75) {$2+2+2+2$};
	
\end{tikzpicture}
\caption{Non-separable graphs realizations of degree sequences of degree sum $8$}
\label{part4}
\end{figure}
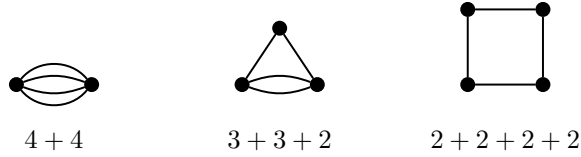

\subsection{Structure of the coefficients of $\RR_n$}
Let us fix an integer $n\geq 3$. Recall the expression of $\RR_n$
\begin{align}\label{eq:Rn}
\RR_n =  \underset{\balpha \in I_{n}}{\sum} c_{\boldsymbol{\alpha}}^{(n)} X_{\balpha}
\end{align}
where $I_n$ is the set of sequences of integers $\balpha = (\alpha_1,\dots, \alpha_r)$ with $r\geq 3$ such that  $\sum_{k=1}^r \alpha_k = 2n$, $\forall k \in \{1,\ldots,r\}, \; 2\leq \alpha_k\leq n-1$ and $\alpha_1 \geq \cdots \geq \alpha_r\geq 2$. Since we only want to deal with coefficents $c_{\boldsymbol{\alpha}}^{(n)}$ that are not zero, we define $I_n^*$ as the set of elements $\balpha$ of $I_n$ such that $c_{\balpha}^{(n)} \neq 0$, so that 
\begin{align}\label{eq:form_rn}
\RR_n =  \underset{\balpha \in I^*_{n}}{\sum} c_{\boldsymbol{\alpha}}^{(n)} X_{\balpha}
\end{align}

Given an element $\balpha$ of $I_n^*$, we denote by $\LL_{n+1}(\balpha)$ (resp. $\HH_{n+1}(\balpha)$) the set of sequences $\bbeta \in I^*_{n+1}$ such that the coefficient of the monomial $X_{\bbeta}$ in $\L(X_{\balpha})$ (resp. $\H(X_{\balpha})$) is non zero. Furthermore, we denote by $\TT_{n+1}(\balpha)$ the union $\LL_{n+1}(\balpha) \cup \HH_{n+1}(\balpha)$. Let us note that, by virtue of \cite[Lemma III.13]{MPS24}, the latter union is disjoint. 

By analysing the proof of \cite[Lemma III.26]{MPS24}, we have the following result:
\begin{lemma}
\begin{align}\label{eq:IT}
I_{n+1}^* & = \underset{\balpha \in I_n^*}{\bigcup} \TT_{n+1}(\balpha).
\end{align}
\end{lemma}
Equality \eqref{eq:IT}states that the polynomial $A_{n+1}$ does not bring more monomials to $\RR_{n+1}$ than the polynomials $\L(\RR_n)$ and $\H(\RR_n)$ .

\begin{definition}
Define $\dsng(n)$ as the set of finite sequences ${\bf d} = (d_1, \ldots,d_r)$ with $r\geq 3$, such that:
\begin{enumerate}[(i)]
\item $d_1 \geq \cdots \geq d_r \geq 2$ ;
    \item $\underset{k=1}{\overset{r}{\sum}} d_k = 2n$ ;
    \item $d_1 \leq d_2 + \cdots + d_r - 2r +4$.
\end{enumerate}
\end{definition}
By Hakimi's Theorem \ref{hakimi_char_nonsep}, the elements $(d_1,\ldots, d_r)$ of $\dsng(n)$ are degree sequence of degree sum $2n$ with a non-separable graph realization. Since the only degree sequence $(d_1,d_2)$ of sum $2n$ is $(n,n)$ (using Theorem \ref{hakimi_char_nonsep}), we deduce that \[|\dsng(n)| = d_{ns}(n)-1.\]

As a consequence of the previously introduced notations, we now formulate below the exact result we are proving in this paper, which implies Proposition \ref{main}.
\begin{proposition}\label{mainbis}
    Let $n$ be an integer greater than $2$. Then $I_n^* = \dsng(n)$.
\end{proposition}

As an illustration, for $n=3$ one has $\dsng(3) = \bigl \{ (2,2,2) \bigr  \}$ and $\RR_3 = -2X_2^3$. The degree sequence $(2,2,2)$ corresponds to the monomial $X_2^3$ (see also Figure \ref{part3} for an associated non-separable graph realization). For $n=4$, one has $\dsng(4) = \bigl \{ (3,3,2), (2,2,2,2) \bigr  \}$ and $\RR_4 = -12X_2X_3^2+6X_2^4$. The degree sequence $(3,3,2)$ (resp. $(2,2,2,2)$) corresponds to the monomial $X_2X_3^2$ (resp. $X_2^4$). Non-separable graph realizations can also be seen in Figure \ref{part4}.

\medskip

For $d_1, \ldots,d_r$ positive integers, we will denote by $\{\{ d_{1}, \ldots,d_r\}\}$ the ordered sequence $(d_{\sigma(1)}, \ldots,d_{\sigma(r)})$ where $\sigma$ is a permutation of $\{1,\ldots,r\}$ such that $d_{\sigma(1)} \geq \cdots \geq d_{\sigma(r)}$.

\section{Proof of Proposition \ref{main}}
Let $n$ be an integer greater than $1$. To prove Proposition \ref{mainbis}, we show that $\dsng(n+1)$ has the same induction relation than $I_{n+1}^*$: equality \eqref{eq:IT}, that is
\begin{align}\label{eq:final}
    \dsng(n+1) & = \underset{\balpha\in \dsng(n)}{\bigcup} \TT_{n+1}(\balpha)
\end{align}
From there, since $I_2^* = \dsng(2)$, equalities \eqref{eq:final} and \eqref{eq:IT} combined with a straightforward induction yield
\begin{align*}
    I_{n+1}^* = \dsng(n+1).
\end{align*}
Thus, we divide the proof in two lemmas.
\begin{lemma}\label{proof1}
    For every sequence $\mathbf{d} \in \dsng(n+1)$, there exists a sequence $\balpha \in \dsng(n)$ such that $\mathbf{d} \in \TT_{n+1}(\balpha)$.
\end{lemma}
\begin{proof}
Let ${\bf d} = (d_1,\ldots,d_r)$ be a degree sequence lying in $\dsng(n+1)$.

If $d_1 = \ldots = d_r = 2$ then since 
\begin{align*}
    2r =\sum_{k=1}^r d_k  = 2(n+1)
\end{align*}
one has $r=n+1$, and  
\begin{align*}
2 \leq 2 \times(r-1) - 2r+4 = 2. 
\end{align*}
hence such a $\bf d$ belongs to $\dsng(n+1)$. Then, $\boldsymbol{\alpha} = (2,\ldots, 2) \in \N^{n}$ verifies ${\bf d} \in \TT_{n+1}(\boldsymbol{\alpha})$ since
\begin{align*}
    \H(X_{\boldsymbol{\alpha}}) = \H(X_2^n) = -n X_2^{n+1} = -n X_{{\bf d}}
\end{align*}
and, what's more, $\balpha \in \dsng(n)$.

Suppose that there exists $k \in \{1,\ldots,r\}$ such that $d_k >2$. Then $d_1 \geq 3$ and $d_2 \geq 3$. Indeed, since $d_1\geq d_k$, we get $d_1 \geq 3$. Finally, one cannot have $d_2 = 2$, since otherwise $d_3 = \cdots = d_r = 2$ and 
\[d_1> 2 = d_2 + \cdots+ d_r -2r+4.\]
Hence $d_2\geq 3$.

Define $\boldsymbol{\alpha} := (d_1-1, d_2-1, d_3, \ldots, d_r)$. Then $\boldsymbol{\alpha}$ is in $\dsng(n)$. Indeed, since $\bf d$ belongs to $\dsng(n+1)$, one has 
\begin{align*}
    &d_1 + d_2 +\sum_{k=3}^r d_k  = 2n+2 ,\\
    &d_1 \leq d_2 + \cdots +d_r -2r+4,
\end{align*}
hence
\begin{align*}
    &d_1-1 + d_2-1 +\sum_{k=3}^r d_k  = 2n ,\\
    &d_1-1 \leq d_2-1+d_3 + \cdots +d_r -2r+4.
\end{align*}
Finally, ${\bf d}$ belongs to $\TT_{n+1}(\boldsymbol{\alpha})$, since 
\begin{align*}
    \L(X_{\boldsymbol{\alpha}}) & = \underset{1\leq i<j \leq r}{\sum} X_{\alpha_1} \cdots X_{\alpha_{i+1}} \cdots X_{\alpha_{j+1}}\cdots X_{\alpha_{r}} \\
    & = X_{\bf d} + \underset{2\leq i<j \leq r}{\sum} X_{\alpha_1} \cdots X_{\alpha_{i+1}} \cdots X_{\alpha_{j+1}}\cdots X_{\alpha_{r}}.
\end{align*}
That concludes.
\end{proof}
From Lemma \ref{proof1} above, we infer a first inclusion:
\begin{align}\label{eq:inclu1}
    \dsng(n+1) \subseteq \underset{\balpha \in \dsng(n)}{\bigcup}\TT_{n+1}(\balpha).
\end{align}
We will now move on to the second lemma.
\begin{lemma}\label{proof2}
    Let $\balpha$ be an element of $\dsng(n)$. Then $\TT_{n+1}(\balpha) \subseteq \dsng(n+1)$.
\end{lemma}
\begin{proof}
Let $\bf d$ be an element of $\TT_{n+1}(\balpha)$. Suppose that $\bf d$ is in $\LL_{n+1}(\balpha)$. Then, recalling \eqref{def:L}, there exist $i < j$ such that
\begin{align*}
{\bf d} = \{\{ \alpha_1, \ldots, \alpha_{i}+1, \ldots,\alpha_{j}+1, \ldots, \alpha_r \}\}.
\end{align*}
Denote by $(d_1, \ldots, d_r)$ the ordered sequence associated to $\bf d$. If $d_1 = \alpha_1$ then since 
\begin{align*}
    \alpha_1 \leq \alpha_2 + \cdots + \alpha_r - 2r +4 \leq \alpha_2 + \cdots + \alpha_r - 2r +4+2,
\end{align*}
one gets
\begin{align*}
    \alpha_1 \leq \alpha_2 + \cdots +(\alpha_i +1)+ \cdots +(\alpha_j+1)+\cdots + \alpha_r - 2r +4,
\end{align*}
hence
\begin{align*}
    d_1 \leq d_2 + \cdots+ d_r - 2r +4.
\end{align*}
Now, if $d_1 =\alpha_i+1$, then one has
\begin{align*}
    \alpha_i+1 &\leq \alpha_1+1 \leq \underset{k=2}{\overset{r}{\sum}} \alpha_k - 2r +4+1\\
    & \leq \alpha_1 + (\alpha_j+1)+\underset{\underset{k\neq i,j}{k=2}}{\overset{r}{\sum}} \alpha_k -2r +4 = d_2 +\cdots +d_r -2r +4.
\end{align*}
Suppose now that $\bf d$ belongs to $\HH_{n+1}(\boldsymbol{\alpha})$. Then, similarly as before, looking closely to \eqref{def:H}, there exist $j$, $l$ such that 
\begin{align*}
    {\bf d} = \{\{ \alpha_1, \ldots, \alpha_{j-1}, 1+l, 1+\alpha_j-l,\alpha_{j+1},\cdots,\alpha_r \}\}.
\end{align*}
If $j\neq 1$, then $d_1 = \alpha_1$ and since
\begin{align*}
    \alpha_1 & \leq \alpha_1 + \cdots + \alpha_r -2r+4= (1+l)+(1+\alpha_k-l)+\underset{\underset{k\neq j}{k=2}}{\overset{r}{\sum}} \alpha_k - 2(r+1) +4,
\end{align*}
one gets
\begin{align*}
    d_1 \leq d_2 + \cdots+ d_{r+1} - 2(r+1) +4.
\end{align*}
If $j=1$ then there are three cases. If $d_1 = \alpha_2$ then
\begin{align*}
    \alpha_2 &\leq \alpha_1 \leq \alpha_2+\cdots+\alpha_r-2r+4 \\ 
    & \leq (\alpha_1+2)+\alpha_3+\cdots+\alpha_r-2(r+1)+4 \\
    & \leq (1+l)+(1+\alpha_1-l)+\alpha_3+\cdots+\alpha_r-2(r+1)+4.
\end{align*}
If $d_1 = 1+l$ then one has
\begin{align*}
    1+l &\leq \alpha_1 \leq \alpha_2+\cdots+\alpha_r-2r+4 \\
    &\leq  \alpha_2+\cdots+\alpha_r+(1+\alpha_1-l)-2-2r+4,
\end{align*}
since $1+\alpha_1-l \geq 2$. The case $d_1 = 1+\alpha_1-l$ is handled in a similar fashion. 
\end{proof}
From Lemma \ref{proof2}, we get now the other inclusion
\begin{align}\label{eq:inclu2}
    \underset{\balpha \in \dsng(n)}{\bigcup} \TT_{n+1}(\balpha) & \subseteq \dsng(n+1).
\end{align}
thus concluding the proof of Proposition \ref{main}.

\section{Acknowledgments}

The author warmly thanks James Sellers for his decisive help in the realisation of this project. The author is funded by the French Community of Belgium with a FRIA grant from the FRS-FNRS.

\bibliographystyle{plain}

\end{document}